\documentclass[12pt]{amsart}
\usepackage{amsmath}
\usepackage{amsfonts,amssymb}
\usepackage{amsthm}
\usepackage[matrix,arrow,curve]{xy}
\numberwithin{equation}{section}
\theoremstyle{plain}
\newtheorem{theorem}{Theorem}[section]
\newtheorem{lemma}[theorem]{Lemma}
\newtheorem{predl}[theorem]{Proposition}
\newtheorem{corollary}[theorem]{Corollary}
\newtheorem{conjecture}[theorem]{Conjecture}
\theoremstyle{definition}
\newtheorem{definition}[theorem]{Definition}
\newtheorem{remark}[theorem]{Remark}

\newcommand{\C}{\mathbb C}
\newcommand{\N}{\mathbb N}
\newcommand{\Z}{\mathbb Z}
\newcommand{\Q}{\mathbb Q}

\newcommand{\A}{\mathbb A}
\renewcommand{\P}{\mathbb P}
\renewcommand{\AA}{\mathcal A}

\newcommand{\D}{\mathcal D}

\newcommand{\EE}{\mathcal E}
\newcommand{\FF}{\mathcal F}

\newcommand{\LL}{\mathcal L}

\renewcommand{\O}{\mathcal O}

\renewcommand{\k}{\mathsf k}

\newcommand{\modd}{\mathrm{mod{-}}}

\newcommand{\coh}{\mathrm{coh}}

\newcommand{\Perf}{\mathrm{Perf}}

\newcommand{\ch}{\mathrm{ch}}

\newcommand{\xra}{\xrightarrow}
\renewcommand{\le}{\leqslant}
\renewcommand{\ge}{\geqslant}

\newcommand{\bul}{\bullet}

\DeclareMathOperator{\Hom}{\textup{Hom}}

\DeclareMathOperator{\Ext}{\textup{Ext}}

\DeclareMathOperator{\Pic}{\mathrm{Pic}}
\DeclareMathOperator{\End}{\mathrm{End}}

\DeclareMathOperator{\rank}{\mathrm{rank}}

\usepackage{array}

\begin{document}

\title[On exceptional collections on del Pezzo surfaces]{On full exceptional collections of line bundles on del Pezzo surfaces}

\author{Alexey ELAGIN}
\thanks{The research was carried out at the IITP RAS at the expense of the Russian
Foundation for Sciences (project ¹ 14-50-00150).}
\address{Institute for Information Transmission Problems (Kharkevich Institute), Moscow, RUSSIA\\
National Research University Higher School of Economics, Moscow, RUSSIA}
\email{alexelagin@rambler.ru}

\author{Valery LUNTS}
\address{Indiana University, Bloomington, USA}
\email{vlunts@indiana.edu}

\begin{abstract}We prove that any numerically exceptional collection of maximal length, consisting of line bundles,  on a smooth del Pezzo surface is a standard augmentation in the sense of L.\,Hille and M.\,Perling. 
We deduce that any such collection is exceptional and full.
\end{abstract}

\maketitle

\tableofcontents

\section{Introduction}

In this work we study exceptional collections of line bundles on surfaces. 
Let~$X$ be a smooth projective variety over a field~$\k$.
Recall that objects $\EE_1,...,\EE_n$ in
the derived category $\D^b(\coh (X))$ form a {\it full exceptional collection} if
\begin{enumerate}
\item $\Hom  (\EE_i,\EE_i[s])=\k$ if $s=0,$ and is zero otherwise;

\item $\Hom (\EE_i,\EE_j[s])=0$ for all $s$ if $j<i$;

\item $\D^b(\coh (X))$ is the smallest full strict triangulated subcategory of $\D^b(\coh (X))$
that contains $\EE_1,...,\EE_n$.
\end{enumerate}

An exceptional collection $(\EE_1,...,\EE_n)$
is {\it strong} if in addition one has

\begin{enumerate}
\item[(2')]  $\Hom (\EE_i,\EE_j[s])=0$ for $s\ne 0$ and all $i,j$.
\end{enumerate}

Having a full exceptional collection in the derived category of coherent sheaves is a nice (but rare) property of an algebraic  variety $X$. It allows one to ``express'' any sheaf (or object of the derived category) on $X$ via the objects of the exceptional collection. For instance, sheaves $\O_{\P^2},\O_{\P^2}(1),\O_{\P^2}(2)$ form a full exceptional collection in $\D^b(\coh(\P^2))$. It follows that for any coherent sheaf $\FF$ on $\P^2$ there exists a bounded complex whose terms are direct sums of $\O_{\P^2}, \O_{\P^2}(1)$ and $\O_{\P^2}(2)$ and whose only nontrivial cohomology is $\FF$.
Varieties with a full exceptional collection of $n$ objects in derived category obey various cohomological restrictions. For instance, their Hodge numbers are located on the diagonal, their Hochschild homology is trivial: $HH_0=\k^n$, $HH_i=0$ otherwise, their Grothendieck group $K_0$ is a lattice generated by classes of exceptional objects. 
Given a full exceptional collection $\EE_1,...,\EE_n$, one can construct a differential graded (or DG)
algebra~$\AA$ with cohomology $\End^*(\oplus_i \EE_i)$. 
By a theorem of B.\,Keller, 
there is an equivalence of categories 
\begin{equation*}
\D^b(\coh (X))\cong \Perf(\AA)
\end{equation*}
where $\Perf(\AA)$ is the homotopy category of right perfect $\AA$-DG-modules. 
This equivalence is especially valuable if the collection  $\EE_1,...,\EE_n$ is strong. Under this assumption one can take the algebra $\AA$ to be an ordinary finite-dimensional associative algebra $\End(\oplus_i \EE_i)$, it is a path algebra of some ordered quiver. The equivalence $\D^b(\coh (X))\cong \D^b(\modd \AA)$ in this case 
provides a connection between the geometry of $X$ and the representation theory of $\AA$. This equivalence introduces a non-standard T-structure on the category $\D^b(\coh(X))$ thus making the use of tilting theory possible. 

Hence one would like to know 
which varieties possess full exceptional collections. 
Among such varieties are projective spaces (A.\,Beilinson, \cite{Be}), Grassmann varieties and quadrics (M.\,Kapranov, \cite{Ka}), many other homogeneous spaces (M.\,Kapranov, \cite{Ka}; A.\,Kuznetsov, A.\,Polishchuk and A.\,Samokhin, \cite{Ku}, \cite{KP}, \cite{PS}, \cite{Sa}),
del Pezzo surfaces (D.\,Orlov, \cite{Or}), toric Fano 3-folds (H.\,Uehara, \cite{Ue}), some other Fano 3-folds (D.\,Orlov, \cite{Or2}; A.\,Kuznetsov, \cite{Ku2}).
Full exceptional collections on the above varieties consist of vector bundles. Also, Yu.\,Kawamata proved that any toric variety has a full exceptional collection, \cite{Kaw}. 

Full strong exceptional collections of vector bundles have been constructed on projective spaces, quadrics, Grassmann varieties, del Pezzo surfaces, toric Fano 3-folds.
It was conjectured by A.\,King (see \cite{K1}) that every smooth toric variety has a full strong exceptional collection of line bundles. In \cite{HP1} L.\,Hille and M.\,Perling described a smooth toric surface that does not have such a collection (hence producing a counter example). 

It is believed that any variety with a full exceptional collection of objects in the bounded derived category is rational. The converse is not true. For example, let $X$ be a blow-up of $\P^3$ in a smooth curve $C$ of positive genus. Then $\D^b(\coh(X))$ has a semiorthogonal component equivalent to $\D^b(\coh(C))$ and hence $\D^b(\coh(X))$ cannot be generated by an exceptional collection.

Let us recall an old conjecture of the second author. 

\begin{definition} We define \emph{Tate} (or \emph{strongly rational}) varieties over $\k$ by induction on the dimension. First we declare an empty variety to be Tate.  Suppose that we have defined Tate varieties of dimension $<d$. Then we say that a variety $Y$ of dimension $d$ is Tate if it has a finite decomposition into
locally closed subvarieties which are either of the following
\begin{enumerate}
\item Tate varieties of dimension $<d$,
\item open subsets $U\subset \A^d$ such that the complement $\A^d\backslash U$ is
Tate.
\end{enumerate}
\end{definition}

Examples of Tate varieties include toric varieties and (partial) flag varieties.
A smooth projective curve is Tate if and only if it is rational.
Classification of surfaces implies that the same is true for surfaces: a smooth projective surface is Tate if and only if it is rational.

\begin{conjecture}
\label{conj_Tate} 
For a smooth projective variety $X$ the category  $\D^b(\coh (X))$ has a full exceptional collection if and only if $X$ is a Tate  variety.
\end{conjecture}

The above conjecture is easy to prove for curves. The ``if'' direction of
the conjecture for surfaces is also easy: moreover, on every rational surface
there exists a full exceptional collection of line bundles. The ``only if''
direction for surfaces is only known if one assumes that the
exceptional collection consists of line bundles and is in addition strong
(M.\,Brown and I.\,Shipman, \cite{BS}). Beyond dimension two the conjecture looks very hard in either direction.

Besides classification of varieties with a full exceptional collection in the derived category, one can try to classify all full exceptional collections on a given variety. The most remarkable (and, unfortunately, the only substantial known to the authors) result in this direction is a theorem by S.\,Kuleshov and D.\,Orlov~\cite{KO}. It claims that  
any full exceptional collection on a del Pezzo surface can be obtained from any other one by a sequence of mutations. In other words, the action of the braid group on the set of full exceptional collections, given by mutations,  is transitive. The similar result was obtained for three-block full exceptional collections on del Pezzo surfaces by B.\,Karpov and D.\,Nogin in~\cite{KN}: the action of the braid group of three strands on the set of three-block exceptional collections on a del Pezzo surface is transitive.

In this work we also deal with classification of full exceptional collections in the derived category of a fixed variety, but in a quite special setting.
We study exceptional collections of line bundles on surfaces. 
The paper \cite{HP2} by Hille and Perling contains the first systematic study of 
full exceptional collections of line bundles on surfaces. In particular, the authors of loc.\,cit.\,introduce the notion of a \emph{standard augmentation} of an exceptional collection, which we recall next.

Let $Y$ be a smooth surface and let $p\colon X \to Y$ be a blowup of a point $P\in Y$ with the corresponding $(-1)$-curve $E\subset X$. Let 
\begin{equation}\label{below} (\O _Y(D_1),...,\O _Y(D_n))
\end{equation}
be a collection of line bundles on $Y$. For some $1\le i\le n$ consider the 
collection 
\begin{equation}\label{stand-augmen}
(\O _X(p ^*D_1+E),...,\O _X(p ^*D_{i-1}+E),\O _X(p ^*D_i),\O _X (p^*D_i+E),
\O _X(p ^*D_{i+1}),...,\O _X(p ^*D_n))\end{equation} 
The collection (\ref{stand-augmen}) 
is called an \emph{augmentation} of the collection
(\ref{below}). If collection (\ref{stand-augmen}) is a full exceptional collection on $X$, then so is collection
(\ref{below}), and vice versa. A collection is called a \emph{standard augmentation} if it is obtained by a series of augmentations from a full exceptional collection on $\P^2$ or a Hirzebruch surface. 

Standard augmentations give us many examples of full exceptional collections of line bundles. 
It is not known whether every full exceptional collection of line bundles on a surface  is a standard augmentation. By \cite[Theorem 8.1]{HP2} this is so for strong collections on toric surfaces. By \cite[Main Theorem 3]{HI}, this is so for exceptional collections (not necessarily strong) on  toric surfaces of Picard rank $3$ or $4$. 
In this paper we prove that this is 
true for exceptional collections on del Pezzo surfaces: 
\begin{theorem}
Any full exceptional collection of line bundles on a smooth del Pezzo surface is a standard augmentation. 
\end{theorem}

Note that we do not suppose that the base field $\k$ is algebraically closed.
Moreover, the same result holds for collections that are exceptional only on numerical level. We recall that a collection of objects $\EE_1,...,\EE_n$ in the derived category 
$\D^b(\coh (X))$ is {\it numerically exceptional} if
\begin{enumerate}
\item $\chi(\EE_i,\EE_i)=1$;

\item $\chi(\EE_i,\EE_j)=0$ for all $i>j$,
\end{enumerate}
where $\chi$ is the Euler pairing: $\chi(\EE_i,\EE_j)=\sum_s (-1)^s\dim\Hom(\EE_i,\EE_j[s])$.
This collection has \emph{maximal length} if the classes $[\EE_i]$ in $K_0(X)$ generate 
$K_0(X)$ modulo numerical equivalence. 

In general, finding a numerically exceptional collection of maximal length in $\D^b(\coh(X))$ is a much easier task then finding a full exceptional collection. For example, there is a classification of smooth projective complex surfaces with $h^1(\O_X)=h^2(\O_X)=0$ admitting a  numerically exceptional collection of maximal length given by C.\,Vial in~\cite{Vi}. It says that a surface $X$ satisfying $h^1(\O_X)=h^2(\O_X)=0$ has a numerically exceptional collection of maximal length if and only if $X$
has such collection of line bundles and if and only if $X$ is one of the following: non-minimal surface, rational surface, surface of general type, Dolgachev surface of types $X_9(2,3),
X_9(2,4), X_9(3,3),X_9(2,2,2)$. Some surfaces of general type (like classical Godeaux surface, see~\cite{BBS}, or Barlow surface, see~\cite{BBKS})
have an exceptional collection of maximal length consisting of line bundles, which is not full. Also, a classical Godeaux surface gives an example of a surface which has 
an exceptional collection of maximal length but has no full exceptional collections.

We prove the following
\begin{theorem}[see Corollary~\ref{corollary_main}]
Let $X$ be a smooth del Pezzo surface. Then any numerically exceptional collection of line bundles having maximal length is a standard augmentation. This collection is exceptional and full.
\end{theorem}

For general rational surfaces the above is not true: see the example in Remark~\ref{remark_Hirz} providing a numerically exceptional collection of maximal length on a Hirzebruch surface, which is not exceptional. We do not know if an arbitrary numerically exceptional collection of maximal length on a del Pezzo surface if full and exceptional.

We would like to thank the anonymous referee who suggested that our results extend to the surfaces over non-algebraically closed fields.
The first author is grateful to Indiana University for their hospitality and inspiring atmosphere.

\section{Preliminaries}

\subsection{Generalities on surfaces}

Throughout this paper we assume that all surfaces are smooth projective connected algebraic surfaces over a field~$\k$ (which is not necessarily algebraically closed). 

Let $X$ be a surface, and $\omega _X=\O_X (K_X)$ its
canonical line bundle.

For a coherent sheaf $\FF$ on $X$ we put 
$$h^i(\FF)=\dim H^i(X,\FF)\quad\text{and}\quad \chi (\FF)=\sum _i(-1)^ih^i(\FF).$$

For a divisor $D$ on $X$ we put 
$$H^i(D)=H^i(X,\O (D)), \quad h^i(D)=\dim H^i(D), \quad\text{and}\quad\chi (D)=\sum_i(-1)^ih^i(D).$$

\subsubsection{Serre duality}

For any coherent sheaves $\EE$ and $\FF$ on $X$ and any $i=0,1,2$ we have
$$\Ext^i(\EE,\FF)\cong \Ext^{2-i}(\FF,\EE\otimes\omega_X)^*$$
(where * denotes the dual vector space). In particular,
for any divisor $D$ on $X$  we have
$$H^i(D)\cong H^{2-i}(K_X-D)^*.$$
Therefore 
$$\chi(D)=\chi(K_X-D).$$

\subsubsection{Riemann-Roch formula}
For any divisor $D$ on $X$ we have
$$\chi (D)=\frac{D\cdot (D-K_X)}{2}+\chi (\O _X).$$

\subsubsection{Hirzebruch surfaces}

The Hirzebruch surface $F_d, d\ge 0$ is defined as the projectivisation of the vector bundle
$\O\oplus \O(d)$ on $\P^1$. It is a rational surface equipped with a ruling $\pi\colon F_d\to\P^1$. Fibers of this ruling are isomorphic to $\P^1$, we denote these fibers by $F$. Clearly, $F^2=0$. The map $\pi$ has a section $B$ with $B^2=-d$, such section is unique for $d>0$. Also $\pi$ has (many) sections $S$ such that $S^2=d$ and $S\cdot B=0$. The Picard group of $F_d$ is a free $\Z$-module with the basis $F,S$ (or $F,B$). One has $S\sim B+dF$. The ruling $\pi\colon F_d\to\P^1$ is unique for $d>0$. The surface $F_0$ is isomorphic to $\P^1\times\P^1$ and has two rulings. The surface $F_1$ is isomorphic to a blow up of $\P^2$ at one point. The canonical class on $F_d$ is represented by $(d-2)F-2S$. The divisor $-K_{F_d}$ is ample if and only if $d\le 1$. Therefore the only Hirzebruch surfaces that are del Pezzo (see Section~\ref{section_delpezzo}) are $F_0$ and~$F_1$.

\subsubsection{Minimal model program for rational surfaces}

Any smooth projective rational surface can be obtained by a sequence of blow-ups from a minimal rational surface. For an algebraically closed field $\k$, any minimal rational surface is either a Hirzebruch surface $F_d$ with $d\ne 1$ or a projective plane $\P^2$. The blow-down of a given rational surface to a minimal one is usually not unique.

\subsubsection{Del Pezzo surfaces}
\label{section_delpezzo}

A del Pezzo surface is a smooth projective surface whose anticanonical class is ample. 
A structure sheaf on a del Pezzo surface is exceptional: one has $h^1(X,\O_X)=h^2(X,\O_X)=0$.
The degree $\deg X$ of a del Pezzo surface $X$ is defined as $K_X^2$, this is a positive integer between $1$ and $9$. It is known (see, for example, Manin's book~\cite{Man})
that any del Pezzo surface over an algebraically closed field is rational.  Any del Pezzo surface over an algebraically closed field with $d=\deg X\ne 8$ is a blow-up of $\P^2$ in $9-d$ generic points. Vise versa, any such blow-up is a del Pezzo surface. For $d=8$ there is also a surface $F_0=\P^1\times \P^1$ which is del Pezzo. 

Suppose $p\colon X\to \P^2$ is a blow-up of $0\le r\le 8$ generic $\k$-rational points $P_1,\ldots,P_r$. Then $X$ is a del Pezzo surface of degree $d=9-r$ (for any field $\k$). Denote by $R_i$ the exceptional divisors of $p$ and by~$H$ the full transform on $X$ of some line on $\P^2$.
Then the Picard group of $X$ admits as a basis $H,R_1,\ldots,R_r$. The intersection form in this basis is given by $H^2=1$,$R_i^2=-1$, $H\cdot R_i=0$, $R_i\cdot R_j=0$ for $i\ne j$.
The canonical class on $X$ is $K_X=-3H+(R_1+\ldots+R_r)$.

\subsubsection{Picard group and $K_0$ group}
Let $K_0(X)$ be the Grothendieck group of a surface~$X$.
This group is equipped with a bilinear form given by the Euler pairing: for coherent sheaves $\EE$ and~$\FF$ one has
$$\chi(\EE,\FF)=\sum_s(-1)^s\dim\Ext^s(\EE,\FF).$$
By Serre duality, the right and the left kernels of $\chi$ are equal, we denote them by $\ker\chi$.

We will need to know that 
\begin{equation}
\label{eq_ranks}
\rank (K_0(X)/\ker\chi)=\rank (\Pic (X)/\equiv) +2,
\end{equation}
where $\equiv$ denotes numerical equivalence. Due to the lack of convenient reference we sketch the proof below.

Denote by $A^{\bul}(X)=\oplus_{k}A^k(X)$ the Chow ring of a surface $X$. Recall that $A^k(X)$ is a group of cycles in $X$ of codimension $k$ modulo rational equivalence, in particular $A^1(X)=\Pic(X)$.
There is an additive map
$K_0(X)\otimes\Q\to A^{\bul}(X)\otimes\Q$,
given by the Chern character. 
Let $N^k(X)$ denote the quotient of $A^k(X)$ modulo numerical equivalence. Since $X$ is a surface, we have
$$N^0(X)=\Z,\qquad N^1(X)=\Pic(X)/\equiv,\qquad N^2(X)=\Z,$$
and $N^k(X)$ vanishes for $k>2$. 
Consider the map
\begin{equation*}
\mathrm{Ch}\colon K_0(X)\otimes\Q\to N^{\bul}(X)\otimes\Q.
\end{equation*}

\begin{lemma}
Let $X$ be a smooth projective surface over a field $\k$. Then the map $\mathrm{Ch}$ induces an isomorphism
\begin{equation}
\label{eq_ch}
(K_0(X)\otimes\Q)/\ker\chi\xra{\sim} N^{\bul}(X)\otimes\Q=\Q\oplus ((\Pic (X)/\equiv)\otimes\Q)\oplus \Q.
\end{equation}
\end{lemma}
\begin{proof}
First we check that $\mathrm{Ch}$ is surjective.
Indeed, for a line bundle $\LL$ one has $\mathrm{Ch}([\LL])=(1,[\LL],\frac12\LL\cdot\LL)\in \Q\oplus N^1(X)\oplus \Q$.
Therefore the image of $\mathrm{Ch}$ generates $N^{\bul}(X)\otimes\Q$ modulo $N^2(X)\otimes\Q=\Q$. Also, for a closed point $P\in X$ one has 
$\mathrm{Ch}([\O_P])=(0,0,-\deg (P))\in \Q\oplus N^1(X)\oplus \Q$, it follows that $\mathrm{Ch}$ is surjective. 

Now we check that $\mathrm{Ch}$ induces an isomorphism~(\ref{eq_ch}). Indeed, the kernel of $\mathrm{Ch}$
consists of the classes whose components of the Chern character are all numerically trivial.
By a modification of Riemann-Roch formula, for coherent sheaves $\EE,\FF$ on $X$ one has

{\tiny 
\begin{multline*}
\chi(\EE,\FF)=ef \chi(\O_X)+\frac12(fc_1(\EE)^2+ec_1(\FF)^2-2c_1(\EE)c_1(\FF))-
\frac12K_X\cdot(ec_1(\FF)-fc_1(\EE))-(fc_2(\EE)+ec_2(\FF))=\\
=ef \chi(\O_X)- \ch_1(\EE)\ch_1(\FF) 
-\frac12K_X\cdot(e\ch_1(\FF)-f\ch_1(\EE))-(f\ch_2(\EE)+e\ch_2(\FF)),
\end{multline*}

}
\noindent
where $e=\rank\EE$ and $f=\rank \FF$.
By linearity this formula also holds for arbitrary  classes in $K_0(X)$.
It can be easily seen that $\chi(\EE,\FF)$ is a non-degenerate pairing on $N^{\bul}\otimes\Q$
applied to $\ch(\EE)$ and $\ch(\FF)$.
Hence the kernel of $\mathrm{Ch}$ is exactly $\ker\chi$ and we get isomorphism~(\ref{eq_ch}).
\end{proof}
Equality (\ref{eq_ranks}) immediately follows from the above Lemma.

\subsection{Generalities on exceptional collections}
\label{section_genexcoll}

For a smooth projective variety $X$ over a field~$\k$, we denote by $\D^b(X)=\D^b(\coh(X))$ the bounded derived category of coherent sheaves on $X$. This is a $\k$-linear $\Hom$-finite triangulated category.
Recall (see~\cite{Bo}) that an object~$\EE$ in
$\D^b(X)$ is \emph{exceptional} if
$\Hom  (\EE,\EE[s])=\k$ if $s=0$ and is zero otherwise.
Objects $\EE_1,...,\EE_n$ in
$\D^b(X)$ form an \emph{exceptional collection} if
\begin{enumerate}
\item every  $\EE_i$ is exceptional;

\item $\Hom (\EE_i,\EE_j[s])=0$ for all $s$ if $j<i$.
\end{enumerate}
This collection $(\EE_1,...,\EE_n)$
is \emph{strong} if in addition one has
\begin{enumerate}
\item[(2')]  $\Hom (\EE_i,\EE_j[s])=0$ for $s\ne 0$ and all $i,j$.
\end{enumerate}
An exceptional collection is \emph{full} if 
\begin{enumerate}
\item[(3)]
$\D^b(X)$ is the smallest full strict triangulated subcategory of $\D^b(X)$
which contains $\EE_1,...,\EE_n$.
\end{enumerate}

There is also a weaker notion called \emph{numerical exceptionality}: this is semiorthogonality on the level of
$K_0$ group. 

It is said that an object $\EE$ in
$\D^b(X)$ is \emph{numerically exceptional} if
$$\chi (\EE,\EE)=1.$$
Objects $\EE_1,\ldots,\EE_n$ in
$\D^b(X)$ form a \emph{numerically exceptional collection} if
\begin{enumerate}
\item every  $\EE_i$ is numerically exceptional;

\item $\chi(\EE_i,\EE_j)=0$ for all $j<i$.
\end{enumerate}

A numerically exceptional collection $(\EE_1,\ldots,\EE_n)$ in $\D^b(X)$ is \emph{numerically full} or \emph{of maximal length} if 
\begin{enumerate}
\item[(3)]
$n=\rank (K_0(X)/\ker\chi)=\rank (\Pic X/\equiv)+2$.
\end{enumerate}

Any exceptional (resp., full exceptional) collection is numerically exceptional (resp., numerically exceptional of maximal length).

Note that there is no analog of \emph{strong} exceptionality on numerical level.

Any exceptional collection $(\EE_1,...,\EE_n)$ in $\D^b(X)$ generates  an infinite sequence $\EE_i, i\in\Z$ of objects of $\D^b(X)$ such that 
$\EE_{i+n}=\EE_i\otimes\omega_X^{-1}$. Up to shifts in the derived category, this sequence is a helix of period $n$, as it is defined in \cite{Bo}.
Note that we do not suppose that the collection $(\EE_1,...,\EE_n)$ is full.
By Serre duality, any subsequence $(\EE_{k+1},\ldots,\EE_{k+n})$ of length~$n$ in the above sequence is also an exceptional collection. 
The similar holds for numerically exceptional sequences.

Suppose that the exceptional collection $(\EE_1,...,\EE_n)$ is full. Then (see \cite[Theorem 4.1]{Bo}) the object $\EE_{k+n+1}$ for $k\ge 0$ can be inductively defined (up to a shift in the derived  category) as a right mutation of $\EE_{k+1}$ over the subcategory $\langle\EE_{k+2},\ldots,\EE_{k+n}\rangle$. Objects $\EE_i$ with $i<0$ can be  defined similarly as left mutations. It follows that any exceptional collection of the form $(\EE_{k+1},\ldots,\EE_{k+n})$ is also full.

\subsection{Left-orthogonal divisors}

In this note we study exceptional sequences of line bundles on surfaces. Clearly, a line bundle $\EE$ on a surface $X$ is exceptional if and only if the 
structure sheaf $\O_X$ is exceptional which is equivalent to $h^1(\O_X)=h^2(\O_X)=0$. 
Likewise, a line bundle $\EE$ on $X$ is numerically exceptional if and only if $\chi(\O_X)=1$.
We note here that these conditions are satisfied for any rational surface and for any del Pezzo surface. On the other hand, there are geometrically irrational surfaces whose structure sheaf is exceptional: for example, Enriques surfaces.

Consider a pair of line bundles $(\O_X(D_1),\O_X(D_2))$ on a surface $X$. Clearly, it is an exceptional pair if and only if $\O_X$ is an exceptional sheaf and
$$h^i(D_1-D_2)=0\qquad\text{for}\qquad i=0,1,2.$$
Likewise, numerical exceptionality of the pair $(\O_X(D_1),\O_X(D_2))$  means $\chi(X,\O_X)=1$ and $\chi(D_1-D_2)=0$.

This motivates the following definition given in \cite{HP2}:

\begin{definition}
Let $D$ be a divisor on a surface $X$.
We say that $D$  is \emph{left-orthogonal} if $$H^i(X,\O_X(-D))=0$$ for all $i$.
We say that  $D$ is \emph{strongly left-orthogonal} if $D$ is left-orthogonal and  $$H^i(X,\O_X(D))=0$$ for $i=1,2$.
We say that $D$ is \emph{numerically left-orthogonal} if $\chi(X,\O_X(-D))=0$.
\end{definition}

The next proposition immediately follows from definitions.
\begin{predl}
\label{prop_gotodiff}
A collection of line bundles
$$(\O_X(D_1),\ldots,\O_X(D_n))$$
on a surface is  exceptional (resp., numerically exceptional) if and only if the sheaf $\O_X$ is exceptional (resp., numerically exceptional) and
for any $1\le i<j\le n$ the divisor $D_j-D_i$ is left-orthogonal (resp., numerically left-orthogonal). 
\end{predl}

\subsection{Toric systems and exceptional collections of line bundles on surfaces}

Let 
$$(\O_X(D_1),\ldots,\O_X(D_n))$$
be a collection of line bundles on a surface.
Following ``astoundingly simple'' idea of L.\,Hille and M.\,Perling, we consider the differences between 
$D_i$ and put
\begin{equation}
\label{eq_AAA}
A_i=
\begin{cases}
D_{i+1}-D_i\quad\text{for}\quad 1\le i\le n-1,\\
(D_1-K_X)-D_n=-K_X-(A_1+\ldots+A_{n-1})\quad\text{for}\quad i=n.
\end{cases}
\end{equation}

Extend the collection $(\O_X(D_1),\ldots,\O_X(D_n))$ infinitely in both directions by the rule $D_{i+n}=D_i-K_X$. Then the equality $D_{i+1}-D_i=A_{i\,\mathrm{mod}\,n}$ holds for any $i\in\Z$ (that explains the definition of $A_n$). 
Clearly, the sequence of differences in the collection $(\O_X(D_{k+1}),\ldots,\O_X(D_{k+n}))$ is  $A_{(k+1)\,\mathrm{mod}\,n}, \ldots, A_{(k+n)\,\mathrm{mod}\,n}$.

\begin{remark}
It follows from Serre duality that, for an exceptional (resp., numerically exceptional) collection $$(\O_X(D_1),\ldots,\O_X(D_n)),$$
the collection 
\begin{equation}
\label{eq_helix}
(\O_X(D_{k+1}),\ldots,\O_X(D_{k+n}))
\end{equation}
is also exceptional (resp., numerically exceptional) for any $k$.
\end{remark}

For future use we reformulate Proposition~\ref{prop_gotodiff} in the following way.
\begin{predl}
\label{prop_ec2ts}
A collection of line bundles $(\O_X(D_1),\ldots,\O_X(D_n))$ on a surface with a toric system $A_1,\ldots,A_n$ is exceptional (resp., numerically exceptional) if and only if the sheaf $\O_X$ is exceptional (resp., numerically exceptional) and for any $1\le i\le j\le n-1$ the divisor $A_i+\ldots+A_j$ is left-orthogonal (resp., numerically left-orthogonal). 
\end{predl}

\begin{definition}
\label{definition_ats}
A sequence of divisors $A_1,\ldots,A_n$ on a surface $X$ is called a \emph{toric system} if $n\ge 3$ and one has
\begin{enumerate}
\item $A_i\cdot A_{i+1}=A_n\cdot A_1=1$ for $i=1,\ldots,n-1$;
\item $A_i\cdot A_j=0$ for $i+1<j$ unless $i=1,j=n$;
\item $A_1+\ldots+A_n\sim -K_X$.
\end{enumerate}
\end{definition}

\begin{remark}
Note that if  $A_1,\ldots,A_n$ is a toric system then 
$A_2,A_3,\ldots,A_n,A_1$ and $A_n,A_{n-1},\ldots,A_1$ are also toric systems.

Informally, one should imagine toric systems as cyclic non-oriented sequences.
\end{remark}

The next proposition is due to L.\,Hille and M.\,Perling, see~\cite[Lemma 3.3 and remarks after it]{HP2}.
\begin{predl}
\label{prop_tsats}
For a numerically exceptional collection 
$(\O_X(D_1),\ldots,\O_X(D_n))$
of line bundles on a surface $X$ with $\chi(\O_X)=1$, the sequence $A_1,\ldots,A_n$ from~(\ref{eq_AAA}) is a toric system if $n\ge 3$.
\end{predl}
Below we give the proof for the convenience of the reader. First we state several lemmas.

\begin{lemma}
\label{lemma_minustwo}
Let $D$ be a numerically left-orthogonal divisor on a surface $X$ with $\chi(\O_X)=1$. Then $D^2=\chi(D)-2$ and $D\cdot K_X=-\chi(D)$. 
\end{lemma}
\begin{proof}
By the Riemann-Roch formula we have
\begin{align*}
\chi(D)&=\chi(\O_X)+\frac{D(D-K_X)}2=1+\frac{D\cdot D-D\cdot K_X}2;\\
0=\chi(-D)&=\chi(\O_X)+\frac{-D(-D-K_X)}2=1+\frac{D\cdot D+D\cdot K_X}2,
\end{align*}
which immediately implies the lemma. 
\end{proof}

\begin{lemma}
\label{lemma_twoslo}
Suppose $D_1,D_2$ and $D_1+D_2$ are numerically left-orthogonal divisors on a surface $X$ with $\chi(\O_X)=1$. Then $\chi(D_1+D_2)=\chi(D_1)+\chi(D_2)$ and $D_1\cdot D_2=1$. 
\end{lemma}
\begin{proof}
By Lemma~\ref{lemma_minustwo} we have
$$\chi(D_1+D_2)=-(D_1+D_2)\cdot K_X=-D_1\cdot K_X-D_2\cdot K_X=\chi(D_1)+\chi(D_2),$$
which proves the first statement. Let us prove the second one. 
By Lemma~\ref{lemma_minustwo} we have 
$$(D_1+D_2)^2=\chi(D_1+D_2)-2=\chi(D_1)+\chi(D_2)-2=D_1^2+D_2^2+2,$$
what implies $D_1\cdot D_2=1$.
\end{proof}

\begin{remark}
The two above lemmas hold for $h^0$ instead of $\chi$ if we assume that the involved numerically left-orthogonal divisors are strongly left-orthogonal.
\end{remark}

\begin{proof}[Proof of Proposition~\ref{prop_tsats}]
Let $(\O_X(D_1),\ldots,\O_X(D_n))$  be a numerically  exceptional collection on $X$ and  $A_1,\ldots,A_n$ be the differences as in~(\ref{eq_AAA}).
We need to check properties (1),(2),(3) of Definition~\ref{definition_ats}.
\begin{enumerate}
\item Equality $A_i\cdot A_{i+1}=1$ follows from Lemma~\ref{lemma_twoslo} for $i=1,\ldots,n-2$ because divisors $A_i,A_{i+1},A_i+A_{i+1}$ are numerically left-orthogonal (see Proposition~\ref{prop_ec2ts}).
To check this equality for $i=n-1$, we consider the numerically  exceptional collection~(\ref{eq_helix}): $(\O_X(D_2),\ldots,\O_X(D_n),\O_X(D_{n+1}))$. Its differences are  $A_2,\ldots,A_n,A_1$ therefore $A_{n-1}\cdot A_n=1$. Similarly, the sequence of differences in the numerically exceptional collection $(\O_X(D_3),\ldots,\O_X(D_n),\O_X(D_{n+1}),\O_X(D_{n+2}))$ is 
$A_3,\ldots,A_n,A_1,A_2$, therefore $A_n\cdot A_1=1$.
\item Suppose $j\ne n$. Then divisors $A_i$, $A_{i+1}+\ldots+A_j$ and $A_i+A_{i+1}+\ldots+A_j$ are numerically left-orthogonal (see Proposition~\ref{prop_ec2ts}) and by Lemma~\ref{lemma_twoslo} we have 
$A_i\cdot (A_{i+1}+\ldots+A_j)=1$. Likewise, we check that $A_i\cdot (A_{i+1}+\ldots+A_{j-1})=1$. It follows now that $A_i\cdot A_j=0$.

If $j=n$ then we assume $i\ne 1$ and we can consider the numerically  exceptional collection $(\O_X(D_2),\ldots,\O_X(D_n),\O_X(D_{n+1}))$ of type~(\ref{eq_helix}). Its differences are $A_2,\ldots,A_n,A_1$ therefore the above arguments can be applied.

\item This follows from the definition of $A_n$.
\end{enumerate}
\end{proof}

\begin{definition}
\label{definition_ts}
The sequence of divisors $A_1,\ldots,A_n$ introduced in~(\ref{eq_AAA}) is called a \emph{toric system} of the collection $(\O_X(D_1),\ldots,\O_X(D_n))$.
\end{definition}

Clearly, an exceptional collection of line bundles can be reconstructed from its toric system uniquely up to a simultaneous twist by a line bundle. That is, for any toric system $A_1,\ldots,A_n$ and a divisor $D_1$ one can consider a sequence $D_1,\ldots,D_n$ of divisors defined by $D_{i+1}=D_1+A_1+\ldots+A_{i}$.

\begin{definition}
We say that a toric system is exceptional (full, strong exceptional, ...) if the corresponding collection of line bundles  is exceptional (full, strong exceptional, ...).
\end{definition}

\begin{remark}
\label{remark_exts}
Note that if  $A_1,\ldots,A_n$ is an exceptional toric system then the toric systems 
$A_2,A_3,\ldots,A_n,A_1$ and $A_n,A_{n-1},\ldots,A_1$ are also exceptional.
That corresponds to the fact that for an exceptional collection $(\O_X(D_1),\ldots,\O_X(D_n))$ the collections 
$$(\O_X(D_2),\ldots,\O_X(D_n),\O_X(D_{n+1}))\quad\text{and}\quad 
(\O_X(-D_n),\ldots,\O_X(-D_1))$$ 
are also exceptional.

The same is true for full exceptional and numerically exceptional toric systems.
\end{remark}

In fact, the notion of a numerically exceptional toric system is useless because the converse to Proposition~\ref{prop_tsats} is true:

\begin{predl}
Any toric system $A_1,\ldots,A_n$ on a surface $X$ with $\chi(\O_X)=1$ is numerically exceptional.
\end{predl}
\begin{proof}
One should check that any divisor $A'=A_k+\ldots+A_l$, where $1\le k\le l\le n-1$, is numerically left-orthogonal. Note that the sequence $A_1,\ldots,A_{k-1},A',A_{l+1},\ldots,A_n$ is also a toric system. Thus, it suffices to consider only the case $k=l$. By Riemann-Roch formula, one has
\begin{multline*}
\chi(-A_k)=1+\frac12(A_k^2+A_k\cdot K_X)=1+\frac12(A_k^2-A_k\cdot \sum_{i=1}^nA_i)=\\
=1+\frac12(A_k^2-A_k^2-A_k\cdot A_{k-1}-A_k\cdot A_{k+1})=1+\frac12(-2)=0.
\end{multline*}
\end{proof}

Therefore, studying numerically exceptional collections of line bundles on rational surfaces is the same as studying toric systems.

\subsection{Toric systems on Hirzebruch surfaces}

Recall that the Picard group of a Hirzebruch surface $F_d$ admits as a basis $F,S$ where $F$ is  a fiber and $S$ is a relatively very ample section. One has $F^2=0, F\cdot S=1, S^2=d$. 
We present a proposition by Hille and Perling~\cite[Proposition 5.2]{HP2}.

\begin{predl}
\label{prop_tsHirzebruch}
On a Hirzebruch surface $F_d$ all toric systems of length four are cyclic shifts of one of the following: 
\begin{gather}
\label{eq_tsHirzebruch1}
F,S+aF,F,S+bF,\quad\text{where}\quad a,b\in\Z, a+b=-d;\\
\label{eq_tsHirzebruch2}
S-\frac d2F,F+a(S-\frac d2F),S-\frac d2F,F-a(S-\frac d2F),\quad\text{where $a\in\Z$ and $d$ even} .
\end{gather}
System~(\ref{eq_tsHirzebruch1}) is exceptional. 
System~(\ref{eq_tsHirzebruch2}) is exceptional only in cases covered by case~(\ref{eq_tsHirzebruch1}): $d=0$ or $d$ even and $a=b=0$.
\end{predl}

\begin{remark}
For $d=0$, toric systems~(\ref{eq_tsHirzebruch1}) and~(\ref{eq_tsHirzebruch2}) are the same up to an automorphism of $\Pic (F_0)$ switching generators $F$ and $S$.
\end{remark}

\begin{remark}
\label{remark_Hirz}
For even $d\ge 2$, there exist toric systems of length $4$ on  $F_d$ that are not exceptional. They give rise to numerically exceptional collections of maximal length that are not exceptional. For instance, such is the next collection on $X=F_2$:
$$(\O_X,\O_X(S-F),\O_X(2S-F),\O_X(3S-2F)),$$
corresponding to toric system~(\ref{eq_tsHirzebruch2}) for $a=1$.
Indeed, the divisor $3S-2F$ is not left-orthogonal:
$$h^2(-3S+2F)=h^0(K_{F_2}+3S-2F)=h^0(S-2F)=h^0(B)=1,$$
where $B$ is the negative section of $F_d$.
\end{remark}

\begin{predl}
On a Hirzebruch surface $X$ all toric systems of type~(\ref{eq_tsHirzebruch1}) are full.
\end{predl}
\begin{proof}
To check that they are full, note that the corresponding exceptional collection is of the form
$$(\O_X, \O_X(F), \O_X(S+kF),\O_X(S+(k+1)F)), k\in\Z.$$
This collection generates two components $\pi^*\D^b(\P^1)$ and $\pi^*\D^b(\P^1)\otimes\O_X(1)$
of a semi-orthogonal decomposition 
$$\D^b(X)=\langle \pi^*\D^b(\P^1), \pi^*\D^b(\P^1)\otimes\O_{X}(1)\rangle,$$
constructed by D.\,Orlov in~\cite{Or}. Here $\pi\colon F_d\to \P^1$ denotes the projection and 
$\O_{X}(1)=\O_X(S+kF)$ is a relatively very ample line bundle. Hence, the above collection is full.
\end{proof}

\begin{corollary}
\label{corollary_hirzebruch}
Let $X$ be  a Hirzebruch surface $F_d$ where $d=0$ or $d$ is odd. Then all toric systems of length four on $X$ are full and exceptional. 
In other words, all numerically exceptional collections of line bundles of maximal length on $X$ are full and exceptional.
\end{corollary}
\begin{proof}
Clearly, for such Hirzebruch surfaces all toric systems are of the form~(\ref{eq_tsHirzebruch1}) and thus full and exceptional.
\end{proof}

\subsection{Augmentations}

In~\cite{HP2} L.\,Hille and M.\,Perling  introduced an operation for ``extending'' exceptional collections and toric systems when blowing up points on a surface. 
Here we recall this operation.

Suppose $A'_1,\ldots,A'_n$ is a toric system on a surface $X'$. Let $p\colon X\to X'$ be the blow-up of a point $P\in X'$, let $E$ be the exceptional divisor of $p$.
Denote $A_i=p^*A'_i$. Consider the following systems of divisors on $X$:
\begin{align}
\label{eq_augmentation1}
&E,A_1-E,A_2,\ldots, A_{n-1},A_n-E;\\
\label{eq_augmentation2}
&A_1,\ldots,A_{m-2}, A_{m-1}-E,E,A_m-E,A_{m+1},\ldots,A_n,\quad\text{for}\quad 2\le m\le n;\\
\label{eq_augmentation3}
&A_1-E,A_2,\ldots, A_{n-1},A_n-E,E.
\end{align}
It is easily checked that the above systems are toric systems.
Any of these systems is called an \emph{augmentation} of the system $A'_1,\ldots,A'_n$.
Note that, up to cyclic shift, all three systems  (\ref{eq_augmentation1})-(\ref{eq_augmentation3}) are of the same form.

Below we collect the main properties of augmentations. Some of them are contained in Proposition 5.5 and Section 6 of \cite{HP2}.
\begin{predl}
\label{prop_augm123}
Let $X'$ be a surface with $h^1(X',\O_{X'})=h^2(X',\O_{X'})=0$. Then
\begin{enumerate}
\item A toric system on $X'$ is exceptional if and only if any of its augmentations is exceptional.
\item A toric system on $X'$ is full and exceptional if and only if any of its augmentations is full and exceptional.
\item Suppose $A'_1,\ldots,A'_n$ is a toric system on $X'$. Then its augmentation (\ref{eq_augmentation2})
is strong exceptional if and only if the system $A'_1,\ldots,A'_n$ is strong exceptional and the point~$P$ is not a base point of the linear system $|A'_k+A'_{k+1}+\ldots+A'_l|$ of divisors on $X'$ for any $1\le k\le l\le n-1$, such that $k\le m, m-1\le l$.
\end{enumerate}
\end{predl}
\begin{proof}
(1)  By Remark~\ref{remark_exts}, it suffices to consider only the system (\ref{eq_augmentation2}): indeed, systems (\ref{eq_augmentation1})-(\ref{eq_augmentation3}) are cyclic shifts of each other. To check that a toric system is exceptional, we use Proposition~\ref{prop_ec2ts}. 
We consider consecutive sums of divisors in a toric system and check if they are left-orthogonal. For system (\ref{eq_augmentation2}), they are of the form
\begin{equation}
\label{eq_sums}
A_i+\ldots+A_j, A_i+\ldots+A_j-E\quad\text{or}\quad E,
\end{equation}
where $j\le n-1$.
We have 
\begin{equation}
\label{eq_augm1}
H^s(X,-(A_i+\ldots+A_j))\cong H^s(X',-(A'_i+\ldots+A'_j))
\end{equation} for all $s$.
Further, consider the exact sequence of sheaves on $X$
$$0\to \O_X(-(A_i+\ldots+A_j))\to \O_X(E-(A_i+\ldots+A_j))\to \O_E(E-(A_i+\ldots+A_j))\to 0.$$
Since $\O_E(E-(A_i+\ldots+A_j))\cong \O_E(-1)$, its long exact sequence of cohomology gives isomorphisms 
\begin{equation}
\label{eq_augm2}
H^s(X,E-(A_i+\ldots+A_j))\cong
H^s(X,-(A_i+\ldots+A_j))\cong H^s(X',-(A'_i+\ldots+A'_j))
\end{equation} for all $s$.

Since $E$ is a rational $(-1)$-curve, one has $H^s(X,\O_X)\cong H^s(X,\O_E)$ for all $s$. Then the standard long exact sequence of cohomology implies
$H^s(X,\O_X(-E))=0$ for all $s$ and divisor $E$ is left-orthogonal. 
Therefore, isomorphisms (\ref{eq_augm1}) and (\ref{eq_augm2}) imply that all divisors in (\ref{eq_sums}) are left-orthogonal if and only if divisors $A'_i+\ldots+A'_j$ are left-orthogonal for all $1\le i\le j\le n-1$.

(2) By Remark~\ref{remark_exts}, it suffices to consider any one of three augmented systems (\ref{eq_augmentation1})- (\ref{eq_augmentation3}). We consider system (\ref{eq_augmentation1}).

By the definition, a toric system $A'_1,\ldots,A'_n$ of divisors on $X'$ is full and exceptional if and only if the collection of line bundles on $X'$
$$(\O_{X'},\O_{X'}(D'_2),\O_{X'}(D'_3),\ldots,\O_{X'}(D'_n))$$
is full and exceptional, where $D'_{i+1}=A'_1+\ldots+A'_i$ for $1\le i\le n-1$.
By a theorem of D.\,Orlov (see~\cite{Or}), this is equivalent to the collection 
$$(\O_{E}(-1),\O_{X},\O_{X}(D_2),\O_X(D_3),\ldots,\O_X(D_n))$$
on $X$ being full and exceptional, where $D_i=p^*D'_i$. Its mutation gives the exceptional collection 
$$(\O_{X},\O_X(E),\O_{X}(D_2),\O_X(D_3),\ldots,\O_X(D_n)).$$
The toric system of this collection is $E,A_1-E,A_2,\ldots,A_{n-1},A_n-E$. Thus this system is full and exceptional if and only if $A'_1,\ldots,A'_n$ is.

(3) As above, we should check whether the divisors of (\ref{eq_sums}) are strongly left-orthogonal.  
Clearly, 
\begin{equation}
\label{eq_augm1'}
H^s(X,A_i+\ldots+A_j)\cong H^s(X',A'_i+\ldots+A'_j)
\end{equation} for all $s$. 
Further, 
consider the exact sequence 
$$0\to \O_X(A_i+\ldots+A_j-E)\to \O_X(A_i+\ldots+A_j)\to \O_E(A_i+\ldots+A_j)\cong \O_E\to 0.$$
Its sequence of cohomology implies that 
\begin{equation}
\label{eq_augm2'}
H^2(X,A_i+\ldots+A_j-E)\cong H^2(X,A_i+\ldots+A_j)\cong H^2(X',A'_i+\ldots+A'_j).
\end{equation}
Also it has a fragment 
\begin{multline}
\label{eq_augm3'}
0\to H^0(X,\O_X(A_i+\ldots+A_j-E))\to H^0(X,\O_X(A_i+\ldots+A_j))\to H^0(X,\O_E)\to\\ \to H^1(X, \O_X(A_i+\ldots+A_j-E))\to H^1(X, \O_X(A_i+\ldots+A_j))\to 0.
\end{multline}

Suppose first that the system $A'_1,\ldots,A'_n$ is strong exceptional. Then  by part (1)
the augmented system (\ref{eq_augmentation2}) is  exceptional. We aim to check that it is strong exceptional. First, the divisor $E$ is strongly left-orthogonal. Next, by (\ref{eq_augm1'}) divisors $A_i+\ldots+A_j$ are strongly left-orthogonal.
By (\ref{eq_augm2'}) we have $H^2(X,A_i+\ldots+A_j-E)=0$. Note that the map $H^0(X,\O_X(A_i+\ldots+A_j))\to H^0(X,\O_E)$ in (\ref{eq_augm3'}) is the restriction map
$$H^0(X',\O_{X'}(A'_i+\ldots+A'_j))\to H^0(X',\O_P)=\k.$$
Since $H^1(X, \O_X(A_i+\ldots+A_j))=0$, then $H^1(X, \O_X(A_i+\ldots+A_j-E))= 0$  if and only if $P$ is not a base point of the linear system $|A'_i+\ldots+A'_j|$.
Consequently, 
all divisors in (\ref{eq_sums}) are strongly left-orthogonal  if and only if $P$ is not a base point of the linear system $|A'_i+\ldots+A'_j|$
for any $i\le j$ such that the divisor $A_i+\ldots+A_j-E$ appears as a sum of consecutive divisors  in the row 
$$A_1,\ldots,A_{m-2}, A_{m-1}-E,E,A_m-E,A_{m+1},\ldots,A_{n-1}.$$
Obviously, the last condition is equivalent to $i\le m, j\ge m-1, j\le n-1$.

Now suppose that the augmented toric system (\ref{eq_augmentation2}) 
is strong exceptional. Then by part (1)
the system $A'_1, \ldots, A'_n$ is exceptional. For any sum $A'_i+\ldots+A'_j$, we have that either divisor $A_i+\ldots+A_j$ or $A_i+\ldots+A_j-E$ is strongly left-orthogonal.
It follows from equations (\ref{eq_augm1'})-(\ref{eq_augm3'}) that in both cases 
$H^s(X',A'_i+\ldots+A'_j)=0$ for $s=1,2$. Consequently, the toric system $A'_1, \ldots, A'_n$ is strong exceptional.
\end{proof}

\begin{remark}
Note that a toric system is of maximal length if and only if its augmentations are of maximal length. Indeed, for a blow-up $X\to X'$ of one point, one has $K_0(X)=K_0(X')\oplus \Z$ where Euler form $\chi_X$ on $K_0(X)$ is given by the matrix
$$\begin{pmatrix}\chi_{X'}&0\\ *&1\end{pmatrix},$$
hence  $\rank (K_0(X)/\ker\chi_X)=\rank (K_0(X')/\ker\chi_{X'})+1$. 
\end{remark}

Starting from a full exceptional toric system on a Hirzebruch surface or on $\P^2$, one can consecutively blow-up points and augment the toric system at each step. The resulting toric system will be called a \emph{standard augmentation}. In other words,
\begin{definition}
A toric system $A_1,\ldots,A_n$ on a surface $X$ is called a \emph{standard augmentation} if one of the following holds:
\begin{enumerate}
\item $X$ is $\P^2$ or a Hirzebruch surface and $A_1,\ldots,A_n$ is a full exceptional toric system;
\item there exists a blow-down $X\to X'$ of a $-1$-curve on $X$ such that $A_1,\ldots,A_n$ is an augmentation of some standard augmentation on $X'$.
\end{enumerate}
\end{definition}

\begin{remark}
Note that in (1) of the above definition one needs to consider all Hirzebruch surfaces, not only minimal ones. Indeed, on the non-minimal Hirzebruch surface $F_1$ there exist many toric systems that are not augmentations from $\P^2$, see Proposition~\ref{prop_tsHirzebruch}.
\end{remark}

By Proposition~\ref{prop_augm123},  a standard augmentation is full and exceptional.  In particular, it follows that on any rational surface over an algebraically closed field there exists a full exceptional collection of line bundles. Note that the standard augmentation is not necessarily strong exceptional, even if the initial toric system on a minimal surface is strong exceptional, see Proposition~\ref{prop_augm123}.(3).

Note that our terminology is a bit different from the one used in~\cite{HP2}: our standard augmentations are called there admissible standard augmentations.

This procedure allows one to construct many full exceptional collections of line bundles on rational surfaces.  Actually, all known full exceptional collections of line bundles on surfaces are standard augmentations. It is proved in \cite{HP2} that any full strong exceptional collection of line bundles on  a toric surface is (maybe after reordering of bundles inside blocks) a standard augmentation. In Section~\ref{section_main} we prove the similar result for del Pezzo surfaces: any numerically exceptional collection of line bundles on  a del Pezzo surface is a standard augmentation.

\section{Main results}
\label{section_main}

Here we prove the main result of the paper: 

\begin{theorem}
\label{theorem_main}
Any toric system of maximal length on a del Pezzo surface~$X$ is a standard augmentation.
\end{theorem}

From this we deduce:

\begin{corollary}
\label{corollary_main}
Any numerically exceptional collection of maximal length of line bundles on a del Pezzo surface~$X$ is a standard augmentation. Also, this collection is full and exceptional.
\end{corollary}

The following observation plays an important role in the proof of Theorem~\ref{theorem_main}.

\begin{predl}
\label{prop_eaugm}
Suppose $A_1,\ldots,A_n$ is a toric system on a surface~$X$. Suppose $A_m=E$ is a $(-1)$-curve on $X$ for some $m$, $1\le m\le n$. Let $p\colon X\to X'$ be the blow-down of~$E$. Then the toric system $A_1,\ldots,A_n$ is an augmentation of some toric system $A'_1,\ldots,A'_{n-1}$ on $X'$. 
\end{predl}
\begin{proof}
For the convenience of notation we will assume that $2\le m\le n-1$.
Consider the elements $A_1,\ldots,A_{m-2},A_{m-1}+E,A_{m+1}+E,A_{m+2},\ldots,A_n\in\Pic X$.
Clearly, any divisor in this sequence has zero intersection with $E$. Therefore, any element in this sequence is a pull-back of some element in $\Pic X'$. Let
\begin{align*}
A_i&=p^*A'_i \quad\text{for} \quad 1\le i\le m-2,\\
A_{m-1}+E&=p^*A'_{m-1}, \\
A_{m+1}+E&=p^*A'_{m}, \\
A_j&=p^*A'_{j-1} \quad\text{for} \quad m+2\le j\le n.\\
\end{align*}
It is directly checked that intersections between  $A_i'$ satisfy relations (1) and (2) from Definition~\ref{definition_ats}. 
Also, 
$$\sum_{i=1}^{n-1} p^*A'_i=\sum_{1\le i\le n, i\ne m} A_i+2E=\sum_{i=1}^n A_i+E=-K_X+E=-p^*K_{X'},$$
therefore $\sum_{i=1}^{n-1} A'_i=-K_{X'}$.
Thus $A'_1,\ldots,A'_{n-1}$ is a toric system on $X'$.
\end{proof}

For the proof of the main theorem we need the following lemma, which is essentially Manin's Theorem IV{.}4{.}3a in~\cite{Man}. We give the proof for the convenience of the reader.
\begin{lemma}
\label{lemma_D-1}
Let $D$ be a numerically left-orthogonal divisor on a del Pezzo surface with $D^2=-1$. Then $D$ is linearly equivalent to a $(-1)$-curve, and $D$ is left-orthogonal.
\end{lemma}
\begin{proof}
By Lemma~\ref{lemma_minustwo}, we have $K_X\cdot D=-1$, $\chi(D)=1$. First, we check that $D$ is linearly equivalent to an effective curve. Indeed, since $(K_X-D)\cdot(-K_X)=-K_X^2-1<0$ and $-K_X$ is ample, we deduce that $h^0(K_X-D)=0$. 
By Serre duality one has 
$h^2(D)=h^0(K_X-D)=0$. 
Therefore $h^0(D)\ge h^0(D)-h^1(D)+h^2(D)=\chi(D)=1$, so we can assume that $D$ is effective.

Let $D=\sum_i k_iE_i$ where $E_i$ are irreducible curves and $k_i\in \N$. Then 
$1=D\cdot (-K_X)=\sum_i k_i(E_i\cdot(-K_X))\ge \sum_i k_i$ because $-K_X$ is ample. Hence $\sum k_i=1$ and $D$ is irreducible. It remains to demonstrate that $D$ is a smooth rational curve. Exact sequence of sheaves 
$$0\to \O_X(-D)\to \O_X\to \O_D\to 0$$
yields an exact sequence  
$$0=H^1(X,\O_X)\to H^1(X,\O_D)\to H^2(X,\O_X(-D)).$$
By Serre duality, $h^2(-D)=h^0(K_X+D)$. Since $(K_X+D)(-K_X)=-K_X^2+1\le 0$ and $-K_X$ is ample, we get that $K_X+D$ is not effective unless $D=-K_X$. The latter case is impossible because $D^2=-1$ and $K_X^2>0$. Hence $h^2(-D)=h^0(K_X+D)=0$ and  $h^1(D,\O_D)=h^1(X,\O_D)=0$. It follows that $D\cong \P^1$, see, for example,~\cite[Exercise IV.1.8b]{Ha}.  Consequently, $D$ is left-orthogonal (see the proof of Proposition~\ref{prop_augm123}).
\end{proof}

\begin{remark}
In general, numerically left-orthogonal divisors on del Pezzo surfaces need not be left-orthogonal. For an example, take the Hirzebruch surface $F_1$ and the divisor $D=-2S+2F=-2B$ on it. Then $D$ is numerically left-orthogonal:
$$\chi(-D)=1+\frac12(D^2+DK_{F_1})=1+\frac12(-4+2)=0,$$
but clearly  $D$ is not left-orthogonal. 
\end{remark}

We recall the next result. It is due to Hille and Perling~\cite{HP2} in the case of an algebraically closed field $\k$ and to C.\,Vial~\cite{Vi} in the general case.

\begin{theorem}[{\cite[Proposition 2.7 or Theorem 3.5]{HP2}}, {\cite[Theorem 3.5]{Vi}}]
\label{theorem_HP}
Let $X$ be a smooth projective surface with $\chi(\O_X)=1$ and with a toric system $A_1,\ldots,A_n$ of maximal length. Then there exists a smooth projective toric surface $Y$ with toric system formed by torus-invariant irreducible divisors $D_1,\ldots,D_n$ such that $A_i^2=D_i^2$ for all $i$.
\end{theorem}

Also we recall that the maximal length   of a numerically exceptional collection on $X$ is $\rank(\Pic (X)/\equiv)+2$. For a del Pezzo surface $X$, the intersection form on the Picard group $\Pic (X)$ is non-degenerate, so we need not distinguish between 
$\Pic(X)$ and $\Pic(X)/\equiv$.

Now we are ready for
\begin{proof}[Proof of Theorem~\ref{theorem_main}]
Let $(\O(E_1),\ldots,\O(E_n))$ be a numerically exceptional collection of maximal length of line bundles on a del Pezzo surface $X$. Let $A_1,\ldots,A_n$ be the corresponding toric system.
The proof is by induction on $n$. Suppose $n=3$ or $4$. Then $\rank\Pic (X)=1$ or $2$ respectively. We claim that $X$ is respectively $\P^2$ or a Hirzebruch surface. Of course, it is obvious if
$\k$ is algebraically closed, but some explanation is needed in the case of arbitrary field $\k$.
Indeed, let $\bar\k$ be an algebraic closure of $\k$ and $\bar X=X\times_{\k}{\bar\k}$ be the scalar extension of $X$. By Theorem 3.3 of C.\,Vial~\cite{Vi}, the natural map 
\begin{equation}
\label{eq_picpic}
\Pic X\to \Pic \bar X
\end{equation}
is an isomorphism preserving the intersection form. Therefore $\bar X$ is a del Pezzo surface with $\rank\Pic \bar X=1$ or $2$, hence $\bar X$ is a $\P^2_{\bar\k}$ or a Hirzebruch surface.
Suppose $n=3$, then $\bar X\cong \P^2_{\k}$.
Isomorphism~(\ref{eq_picpic}) implies that there is a divisor $H$ on $X$ such that $H^2=1$ and $h^0(X,\O_X(H))=h^0(\bar X,\O_{\bar X}(\bar H))=3$. Divisor $H$ gives a map $\phi_H\colon X\to \P_{\k}^2$. Since $\overline{\phi_H}$ is an isomorphism,  $\phi_H$ is also an isomorphism. Now suppose $n=4$, then $\bar X\cong (F_{d})_{\bar\k}$ where $d=0$ or $1$. Isomorphism~(\ref{eq_picpic}) implies that there are effective divisors $F$ and $B$ on $X$ such that $F^2=0, B^2=-d, F\cdot B=1$. Suppose $d=0$. Then $h^0(X,\O_X(F))=h^0(\bar X,\O_{\bar X}(\bar F))=2$, and divisor $F$ gives a map $\phi_F\colon X\to \P_{\k}^1$. Similarly there is a map $\phi_B\colon X\to \P^1_{\k}$ and the map $\phi_F\times \phi_B\colon X\to (\P^1_{\k})^2=(F_0)_{\k}$ gives an isomorphism.
 Suppose $d=1$. Then $B$ is a smooth rational $-1$-curve. Indeed, $\bar B\cong \P^1_{\bar\k}$ since $\bar X$ is a Hirzebruch surface $F_1$. Further, $B$ has a 
$\k$-point because $F\cdot B=1$, consequently $B\cong \P^1_{\k}$. Let $p\colon X\to Y$ be the blow-down of $B$, then $Y$ is a smooth del Pezzo surface with $\Pic Y=\Z$. Since 
$(B+F)\cdot B=0$, there is a divisor $H$ on $Y$ such that $p^*H\sim B+F$. One has $h^0(Y,\O_Y(H))=h^0(\bar Y,\O_{\bar Y}(\bar H))=3$, and by the above arguments we deduce that $Y\cong \P^2_{\k}$. It follows that $X$ is a blow-up of $\P^2_{\k}$ at a $\k$-point, that is, $X\cong (F_1)_{\k}$.

It remains to demonstrate that the toric system $A_1,\ldots,A_n$ on $X$ is full and exceptional, what follows from Corollary~\ref{corollary_hirzebruch}.

So we can suppose $n\ge 5$. 
By Theorem~\ref{theorem_HP}, there exists a toric surface $Y$ with torus invariant irreducible  divisors $D_1,\ldots,D_n$ such that $D_i^2=A_i^2$ for any $i$. Since $Y$ is not  a minimal surface, it possesses a $-1$-curve $R$. This curve has to be torus invariant (otherwise it is movable and thus $R^2\ge 0$), hence $R=D_i$ for some~$i$. Therefore, $A_i^2=D_i^2=-1$.
By Lemma~\ref{lemma_D-1} divisor $A_i$ is linearly equivalent to a $(-1)$-curve $E$.
By Proposition~\ref{prop_eaugm}, the toric system $A_1,\ldots,A_n$ is an augmentation of some toric system on the blow-down of~$E$. This blow-down is also a del Pezzo surface by~\cite[Corollary IV{.}2{.}8]{Man} and one can proceed by induction.
\end{proof}

\begin{proof}[{Proof of Corollary~\ref{corollary_main}}]
Let $(\O_X(E_1),\ldots,\O_X(E_n))$ be a numerically exceptional collection on $X$ of maximal length.
By Proposition~\ref{prop_tsats}, it corresponds to a toric system $A_1,\ldots,A_n$ of maximal length. By Theorem~\ref{theorem_main}, this toric system is a standard augmentation. It means that $A_1,\ldots,A_n$ is obtained by a sequence of augmentations from some toric system $A'_1,A'_2,A'_3,A'_4$ on a Hirzebruch surface $X'$ (unless $X=\P^2$, which is a trivial case). Clearly, $X'$ is a del Pezzo surface, hence $X'\cong F_0$ or~$F_1$ (see the proof of Theorem~\ref{theorem_main}). By  Corollary~\ref{corollary_hirzebruch},
the toric system $A'_1,A'_2,A'_3,A'_4$ is full and exceptional. Finally, by Proposition~\ref{prop_augm123} the toric system $A_1,\ldots,A_n$ is also full and exceptional.
\end{proof}


\begin{thebibliography}{99}

\bibitem{Be} A.\,A.\,Beilinson, ``Coherent sheaves on $P^n$ and problems in linear algebra'', \textit{Funktsional.\,Anal.\,i Prilozhen.}, 12(3):68{--}69, 1978.

\bibitem{BBS} C.\,Bohning, H.-C.\,Graf von Bothmer, P.\,Sosna, ``On the derived category of the classical Godeaux surface'', \textit{Advances in Mathematics}, 243:203{-–}231, 2013.

\bibitem{BBKS} C.\,Bohning, H.-C.\,Graf von Bothmer, L.\,Katzarkov, P.\,Sosna, ``Determinantal Barlow surfaces and phantom categories'', \textit{Journal of the EMS}, 17(7):1569{-–}1592, 2015.

\bibitem{Bo}
A.\,I.\,Bondal, ``Representation of associative algebras and coherent sheaves'', {\it Math.\,USSR Izvestiya}, 34(1):23{--}42, 1990.


\bibitem{BS} M.\,Brown, I.\,Shipman, ``The McKay correspondence, tilting equivalences, and
rationality'',  arXiv:1312.3918v3, 2014.


\bibitem{Ha} R.\,Hartshorne, {\it Algebraic geometry}, Springer, New-York, 1981.


\bibitem{HP1} L.\,Hille, M.\,Perling, ``A Counterexample to King's Conjecture'', {\it Compos.\,Math.}, 142(6):1507{--}1521, 2006.

\bibitem{HP2} L.\,Hille, M.\,Perling, ``Exceptional sequences of invertible sheaves on rational
surfaces'', {\it Compos.\,Math.}, 147(4):1230{--}1280, 2011.

\bibitem{HI} A.\,Hochenegger, N.\,Ilten, ``Exceptional Sequences on Rational C*-Surfaces'', {\it Manuscripta Mathematica}, 142:1{--}34, 2013.

\bibitem{Ka}
M.\,M.\,Kapranov, ``On the derived categories of coherent sheaves on some homogeneous
spaces'', \textit{Invent.\,Math.}, 92(3):479{--}508, 1988.


\bibitem{KN} 
B.\,V.\,Karpov, D.\,Yu.\,Nogin, 
``Three-block exceptional collections over Del Pezzo surfaces'', 
\textit{Izv.\,Ross.\,Akad.\,Nauk Ser.\,Mat.}, 62(3):429{--}463, 1998.

\bibitem{Kaw}
Yu.\,Kawamata, ``Derived categories of toric varieties'', \textit{Michigan Math.\,J.},  54:517{--}535, 2006.


\bibitem{K1} A.\,King, ``Tilting bundles on some rational surfaces'', unpublished manuscript, see \textsl{http://www.maths.bath.ac.uk/~masadk/papers/}, 1997.

\bibitem{KO}
S.\,A.\,Kuleshov, D.\,O.\,Orlov, ``Exceptional sheaves on del Pezzo surfaces'', \textit{Russian Acad. Sci. Izv. Math.}, 44(3):479{-–}513, 1995.


\bibitem{Ku2}
A.\,G.\,Kuznetsov, ``An exception set of vector bundles on the varieties $V_{22}$'', \textit{Moscow Univ.\,Math.\,Bull.}, 51(3):35{--}37, 1996.


\bibitem{Ku} A.\,Kuznetsov,  
``Exceptional collections for Grassmannians of isotropic lines'', \textit{Proc.\,Lond.\,Math.\,Soc.}, 97(1):155{--}182, 2008.

\bibitem{KP} A.\,Kuznetsov, A.\,Polishchuk, ``Exceptional collections on isotropic Grassmannians'', arXiv:1110.5607, 2011.


\bibitem{Man}
Yu.\,I.\,Manin, \textit{Cubic Forms}, North-Holland Math. Lib., 1974.

\bibitem{Or2}
D.\,O.\,Orlov, ``Exceptional set of vector bundles on the variety $V_5$'', 
\textit{Moscow Univ.\,Math.\,Bull.}, 46(5):48{--}50, 1991. 


\bibitem{Or}
D.\,O.\,Orlov, ``Projective bundles, monoidal
transformations and derived categories of coherent sheaves'', \textit{Russian Akad.\,Sci.\,Izv.\,Math}, 41:133{--}141, 1993.



\bibitem{PS}
A.\,Polishchuk, A.\,Samokhin, ``Full exceptional collections on the Lagrangian
Grassmannians $LG(4,8)$ and $LG(5,10)$'', \textit{J.\,Geom.\,Phys.}, 61:1996{--}2014, 2011.

\bibitem{Sa}
A.\,V.\,Samokhin, ``The derived category of coherent sheaves on $LG_3^{\C}$'', \textit{Russian Mathematical Surveys}, 56(3):592{--}594,  2001.


\bibitem{Ue}
H.\,Uehara, ``Exceptional collections on toric Fano threefolds and birational geometry'', 
\textit{Internat.\,J.\,Math.}, 25(7), 2014.

\bibitem{Vi}
C.\,Vial, ``Exceptional collections, and the Neron-Severi lattice for surfaces'', 
arXiv:1504.01776v3, 2015.



\end{thebibliography}
\end{document}